\documentclass[11pt,a4paper]{article}
\usepackage{amssymb}
\usepackage{eurosym}
\usepackage{amsfonts}
\usepackage{amsmath}
\usepackage{amsthm}
\usepackage{graphicx}
\usepackage{float}
\usepackage{hyperref}
\usepackage[pagewise]{lineno}

\hypersetup{
	colorlinks=true,
	linkcolor=blue,
	anchorcolor=blue,
	citecolor=blue
}
\newtheorem{theorem}{Theorem}[section]

\newtheorem{conjecture}[theorem]{Conjecture}

\newtheorem{lemma}[theorem]{Lemma}
\newtheorem{proposition}[theorem]{Proposition}
\theoremstyle{definition}
\newtheorem{definition}[theorem]{Definition}
\newtheorem{example}[theorem]{Example}

\newtheorem{remark}[theorem]{Remark}

\renewenvironment{proof}[1][Proof]{\noindent\textbf{#1.} }{\ \rule{0.5em}{0.5em}}
\newenvironment{acknowledgement}{\smallskip{\sc Acknowledgement.}\rm}{\smallskip}
\renewcommand{\theequation}{\thesection.\arabic{equation}}
\allowdisplaybreaks
\input{tcilatex}

\def\func#1{\mathop{\mathrm{#1}}\nolimits}

\def\Xint#1{\mathchoice
{\XXint\displaystyle\textstyle{#1}}%
{\XXint\textstyle\scriptstyle{#1}}%
{\XXint\scriptstyle\scriptscriptstyle{#1}}%
{\XXint\scriptscriptstyle\scriptscriptstyle{#1}}%
\!\int}
\def\XXint#1#2#3{{\setbox0=\hbox{$#1{#2#3}{\int}$ }
\vcenter{\hbox{$#2#3$ }}\kern-.6\wd0}}

\def\oint{\Xint-}

\def\enddoc{\end{document}}

\def\FRAME#1#2#3#4#5#6#7#8
{
 \begin{figure}[H]
 \begin{center}
 \includegraphics[height=#3]{#7}
 \caption{#5}
 \label{#6}
 \end{center}
 \end{figure}
}

\begin{document}
	\title{Conservation of mass for solutions of Leibenson's equation on Riemannian Manifolds}
	\author{Philipp S\"urig}
	\date{September 2026}
	\maketitle
	
	\begin{abstract}
		We consider on a Riemannian manifold $M$ the Leibenson equation \begin{equation*}\label{eqabs}\partial _{t}u=\Delta _{p}u^{q},\end{equation*} where $p>1$ and $q>0$. When $q(p-1)\geq 1$, we prove conservation of mass for solutions of Leibenson's equation assuming only the volume bound $V(x_0, r)\leq \exp\left(C r^{\frac{p}{p-1}}\right)$ for some $x_{0}\in M$ and all large enough $r>0$. When $q(p-1)< 1$, we prove this property assuming $V(x_0, r)\leq Cr^{N}$ and $p>N[1-q(p-1)]$, which matches the threshold in $\mathbb{R}^{n}$ with $N=n$. We also show that solutions on the hyperbolic space $\mathbb{H}^{n}$ have a finite extinction time in the case $q(p-1)< 1$, which implies that the conservation of mass property does not hold.
		Using the conservation of mass result in the case $q(p-1)=1$, we also prove a $L^{p-1}$- Liouville property, which partially answers a conjecture stated by I. Holopainen \cite{holopainen2000sharp}.		
	\end{abstract}
	
	\let\thefootnote\relax\footnotetext{
		The author was funded by the Deutsche Forschungsgemeinschaft (DFG,
		German Research Foundation) - Project-ID 317210226 - SFB 1283.}
	
	\tableofcontents
	
	\section{Introduction}
	
	Let $M$ be a Riemannian manifold. We consider solutions of the non-linear evolution
	equation 
	\begin{equation}
		\partial _{t}u=\Delta _{p}u^{q},  \label{evoeq}
	\end{equation}%
	where $$p>1\quad \textnormal{and}\quad q>0,$$ $u=u(x,t)$ is an unknown non-negative function of $x\in M$, $t\geq0$ and $%
	\Delta _{p}$ is the Riemannian $p$-Laplacian 
	$\Delta _{p}v=\func{div}\left( |\nabla v|^{p-2}\nabla v\right).$
	For the physical meaning of (\ref{evoeq}) see \cite{grigor2024finite, leibenzon1945general, leibenson1945turbulent}. 
	
	The equation (\ref{evoeq}) is also referred to as \textit{Leibenson's equation} or a \textit{doubly non-linear parabolic equation}. 
	
	When $M=\mathbb{R}^{n}$ and if \begin{equation}\label{critical}p>n[1-q(p-1)]\end{equation} it is well-known (cf. \cite{agueh2010large, dibenedetto1993degenerate, vazquez2007porous}) that any solution of (\ref{evoeq}) with non-negative initial function $u(\cdot, 0)=u_0\in L^{1}(M)$ has the property of \textit{mass conservation}, that is, for all $t>0$, \begin{equation}\label{massconsint}\int_M u(x, t)=\int_M u_0(x).\end{equation}
	On the other hand, if \begin{equation}\label{noncrit}p<n[1-q(p-1)],\end{equation} then any solution of (\ref{evoeq}) has a \textit{finite extinction time} in $\mathbb{R}^{n}$, that means there exists $T>0$ such that $$u(\cdot, t)\equiv0\quad \text{for all}~t\geq T.$$	
	Hence, $u$ can not satisfy (\ref{massconsint}) when $u_0\not\equiv0$ and (\ref{noncrit}) holds.
	
	In the present paper we are interested in investigating the property (\ref{massconsint}) for solutions of (\ref{evoeq}) on general Riemannian manifolds.
	
	Let $M$ be a geodesically complete Riemannian manifold. Denote by $\mu$ the \textit{Riemannian measure} on $M$, by $d$ the \textit{geodesic distance} and by $B(x, r)$ the \textit{geodesic ball} of radius $r$ centered at $x$.
	We denote its volume by $$V(x, r)=\mu(B(x, r)).$$
		
	The first main result of the present paper is as follows (cf. Theorem \ref{expmass}).
	
	\begin{theorem}\label{expmassint}
		Assume that, for some $x_0\in M$ and all large enough $r>0$, \begin{equation}\label{expvolint}V(x_0, r)\leq \exp\left(C r^{\frac{p}{p-1}}\right),\end{equation} where $C$ is a positive constant.
		Assume that \begin{equation}\label{degandexcase}q(p-1)\geq 1.\end{equation}
		Let $u$ be a non-negative bounded solution to (\ref{evoeq}) in $M\times [0, \infty)$ with non-negative initial function $u(\cdot, 0)=u_0\in L^{1}(M)\cap L^{\infty}(M)$. Then (\ref{massconsint}) holds for all $t>0$.
	\end{theorem}

For example, (\ref{expvolint}) is satisfied by $\mathbb{R}^{n}$ and the hyperbolic space $\mathbb{H}^{n}$.

In the case $q(p-1)=1$, that is, when (\ref{evoeq}) becomes the \textit{Trudinger equation} $\partial _{t}u=\Delta _{p}u^{1/(p-1)}$ we also construct a class of spherically symmetric manifolds with a volume growth $V(x, r)\simeq \exp\left(C r^{\alpha}\right)$, where $\alpha>\frac{p}{p-1}$ and a solution of the Trudinger equation that does not satisfy (\ref{massconsint}) (cf. Section \ref{appendix}).

The method in the proof of Theorem \ref{expmassint} uses similar arguments as \cite{grigor1986stochastically}, where \textit{stochastic completeness} of manifolds was proved under the condition \begin{equation}\label{conjectureint}\int^{\infty}\frac{rdr}{\log V(x_0, r)}=\infty.\end{equation} Clearly, (\ref{expvolint}) implies (\ref{conjectureint}) when $p=2$. In particular, we prove and utilize a \textit{Davies-Gaffney} type inequality when (\ref{degandexcase}) holds, which was proved in \cite{surig2024sharp} in the case $q(p-1)=1$.

When $p=2$ and $q>1$, that is, (\ref{evoeq}) becomes the \textit{porous medium equation} $\partial _{t}u=\Delta u^{q}$, conservation of mass was proved by G. Grillo, M. Muratori and F. Punzo on \textit{Cartan-Hadamard} manifolds in \cite{grillo2018porous}. The novelty of our result Theorem \ref{expmassint} is that we prove conservation of mass on a general manifold satisfying (\ref{expvolint}) in the range (\ref{degandexcase}), in particular, we allow $q=1$ when $p\geq 2$, that is, (\ref{evoeq}) becomes the $p$-Laplace equation $\partial _{t}u=\Delta_p u$.

The second main result of the present paper is as follows (cf. Theorem \ref{expmasssing}).

\begin{theorem}\label{expmasssingint}
	Assume that, for some $x_0\in M$ and all large enough $r>0$, \begin{equation}\label{expvolsingint}V(x_0, r)\leq Cr^{N},\end{equation} where $C, N$ are positive constants.
	Suppose that \begin{equation}\label{singularint}q(p-1)<1\end{equation} and \begin{equation}\label{pnDint}p>N[1-q(p-1)].\end{equation} Let $u$ be a non-negative bounded solution to (\ref{evoeq}) in $M\times [0, \infty)$ with non-negative initial function $u(\cdot, 0)=u_0\in L^{1}(M)\cap L^{\infty}(M)$. Then (\ref{massconsint}) holds for all $t>0$.
\end{theorem}

Note that $\mathbb{R}^{n}$ satisfies (\ref{expvolsingint}) with $N=n$. Hence, we see that in this case (\ref{pnDint}) matches (\ref{critical}).

A finite extinction time for solutions of (\ref{evoeq}) on manifolds satisfying an euclidean-type Sobolev inequality and under certain conditions on $p$ and $q$ was proved in \cite{surig2024finite}. However, the range obtained in \cite{surig2024finite} was not optimal. In the present paper we improve in Proposition \ref{thmfinexint} the range of $p$ and $q$ which matches (\ref{noncrit}) in the case when $M=\mathbb{R}^{n}$ (cf. Example \ref{exampleRn}).

Moreover, combining Proposition \ref{thmfinexint} with the Poincaré inequality and the Sobolev inequality in $\mathbb{H}^{n}$, we obtain a finite extinction time for solutions of (\ref{evoeq}) in $\mathbb{H}^{n}$ in the whole range (\ref{singularint}), in particular, when also (\ref{critical}) holds (cf. Example \ref{exampleHn}). Recall that in $\mathbb{H}^{n}$ we have $V(x, r)\simeq \exp((n-1)r)$. Hence, this example shows that the range (\ref{critical}) is not universal to obtain mass conservation on manifolds. When $p=2$, that is (\ref{singularint}) becomes $q<1$, finite extinction time in $\mathbb{H}^{n}$ was proved by M. Bonforte, G. Grillo and J.L. Vazquez in \cite{bonforte2008fast}.
Proposition \ref{thmfinexint} extends this property to the Leibenson equation (\ref{evoeq}) with $p>1$.

The structure of the present paper is as follows.

In Section \ref{secweak}, we
define the notion of a weak solution of the Leibenson equation (\ref{evoeq}). We also prove in Section \ref{secweak} in Lemma \ref{Lem1} a Caccioppoli type inequality, which generalizes a Caccioppoli type inequality from \cite{grigor2024finite}.

In Section \ref{intestsec} we prove in Lemma \ref{choicofxi} the aforementioned Davies-Gaffney type inequality, which is then used in the proof of Theorem \ref{expmassint}.

In Section \ref{seccons} we prove our main results Theorem \ref{expmassint} and Theorem \ref{expmasssingint}. In this section we also prove in Proposition \ref{propfps} conservation of mass in the \textit{degenerate} case $q(p-1)>1$ on any manifold with non-negative \textit{Ricci-curvature}, using the fact that we have a \textit{finite propagation speed} in this situation (cf. \cite{grigor2024finite}). 

From Theorem \ref{expmassint} we also obtain in Proposition \ref{propLiou} a $L^{p-1}$- Liouville property, which gives a partial answer to a conjecture stated by I. Holopainen \cite{holopainen2000sharp} (cf. Remark \ref{remholo}). In the linear case $p=2$, this conjecture was proved by A. Grigor'yan in \cite{grigor1989stochastically} by showing that stochastic completeness implies the $L^{1}$- Liouville property.

In Section \ref{secextinct} we discuss finite extinction time by means of the examples $\mathbb{R}^{n}$ and $\mathbb{H}^{n}$.

In Section \ref{appendix} we construct the aforementioned class of spherically symmetric manifolds with a faster volume growth than (\ref{expvolint}) that admits a solution of (\ref{evoeq}) in the case $q(p-1)=1$ that does not satisfy (\ref{massconsint}).

For existence results for weak solutions of (\ref{evoeq}) on Riemannian manifolds, we refer to \cite{surig2026existence}.

We denote by $c, C, C^{\prime}$ positive constants whose value might change at each occurrence.
	
	\begin{acknowledgement}
		The author would like to thank Alexander Grigor'yan for many helpful suggestions.
	\end{acknowledgement}

	\section{Weak subsolutions}
	
	\label{secweak}
	
	We consider in what follows the following non-linear evolution
	equation on a Riemannian manifold $M$:%
	\begin{equation}
		\partial _{t}u=\Delta _{p}u^{q}.  \label{dtv}
	\end{equation} 
	By a \textit{subsolution} of (\ref{dtv}) we mean a non-negative function $u$
	satisfying $$\partial _{t}u\leq\Delta _{p}u^{q}$$
	in a certain weak sense as explained below.
	
	We assume throughout that 
	\begin{equation*}
		p>1\quad\textnormal{and}\quad q>0.
	\end{equation*}%
	
	Let $\mu $ denote the Riemannian measure on $M$. For simplicity of notation,
	we frequently omit in integrations the notation of measure. All integration
	in $M$ is done with respect to $d\mu $, and in $M\times \mathbb{R}$ -- with
	respect to $d\mu dt$, unless otherwise specified.
	
	Let $\Omega$ be an open set in $M$ and $I$ be an interval in $[0, \infty)$.
	
	\begin{definition}
		\normalfont
		We say that a non-negative function $u=u(x, t)$ is a \textit{weak
			subsolution} of (\ref{dtv}) in $\Omega\times I$, if 
		\begin{equation}  \label{defvonsoluq}
			u\in C\left(I; L^{q+1}(\Omega)\right)\quad \textnormal{and}\quad 
			u^{q}\in L_{loc}^{ p}\left(I; W^{1, p}(\Omega)\right),
		\end{equation} where (\ref{dtv}) holds weakly in $\Omega\times I$, which means that for all $t_{1}, t_{2}\in I$ with $t_{1}<t_{2}$, and all non-negative \textit{test functions} 
		\begin{equation}  \label{defvontestsoluq}
			\psi\in W_{loc}^{1, 1+\frac{1}{q}}\left(I;
			L^{1+\frac{1}{q}}(\Omega)\right)\cap L_{loc}^{p}\left(I; W_{0}^{1,
				p}(\Omega)\right),
		\end{equation}
		we have 
		\begin{equation}  \label{defvonweaksolq}
			\left[\int_{\Omega}{u\psi }\right]_{t_{1}}^{t_{2}}+\int_{t_{1}}^{t_{2}}{%
				\int_{\Omega}{-u\partial_{t}\psi+|\nabla u^{q}|^{p-2}\langle\nabla u^{q},
					\nabla \psi\rangle}}\leq 0.
		\end{equation}
	\end{definition}
	
	\textit{Weak supersolutions} and \textit{weak solutions} of (\ref{dtv}) are
	defined analogously.
	
	Existence results for weak solutions of (\ref{dtv}) were obtained in \cite{bogelein2018doubly, coulhon2016regularisation,  ishige1996existence, ivanov1997regularity, lindqvist2026lipschitz} in the Euclidean setting and in \cite{surig2026existence} on manifolds.
	
	Let us also set \begin{equation}\label{delta}\delta=q(p-1)-1.\end{equation}
	
	\begin{lemma}[Caccioppoli-type inequality]
		\label{Lem1} Let $\Omega $ be an open subset of $M$ and let $I$ be an interval in $\mathbb{R}_{+}=[0, \infty)$. Let $u=u\left( x,t\right) $ be a bounded
		non-negative subsolution to \emph{(\ref{dtv})} in a cylinder $\Omega\times I$. Let $\eta $ be a locally Lipschitz non-negative bounded
		function in $\Omega\times I$ such that $\eta(\cdot, t)$ has
		compact support in $\Omega $ for any $t\in I$. Fix some real $\sigma $
		such that 
		$\sigma > q(p-1)$ and set $\lambda=\sigma-\delta>1$. 
		Choose $t_{1},t_{2}\in I$ such that $t_{1}<t_{2}$. Then%
		\begin{equation}
			\left[ \int_{\Omega }u^{\lambda }\eta^{p} \right] _{t_{1}}^{t_{2}}+c_{1}
			\int_{t_{1}}^{t_{2}}
				\int_{\Omega}\left\vert \nabla \left( u^{\sigma/p }\eta \right) \right\vert
			^{p} \leq \int_{t_{1}}^{t_{2}}
				\int_{\Omega}\left[c_2u^{\sigma}|\nabla \eta|^{p}+pu^{\lambda}\eta^{p-1}\partial _{t}\eta\right],
			\label{vetacor}
		\end{equation}
		where $c_{1}, c_2$ are constants depending on $p, q$ and $\sigma$. 
	\end{lemma}

\begin{proof}
In the case $\sigma\geq pq$, the proof works as in \cite{grigor2024finite}. Hence it remains to prove (\ref{vetacor}) in the case $q(p-1)< \sigma <qp$. Set $a=\frac{\lambda-1}{q}=\frac{\sigma-q(p-1)}{q}\in (0, 1)$. Consider, for any $\varepsilon>0$, the function $f_\varepsilon(s)=(s+\varepsilon)^{a}-\varepsilon^a$. Note that $f_\varepsilon(0)=0$ and $f_\varepsilon$ is a Lipschitz function. Hence, we can test with the function $\psi=f_\varepsilon(u^q)\eta^{p}\theta_\nu$ in Lemma 2.5 in \cite{Grigoryan2024}, where $\theta_\nu$ is the cut-off function of $[t_1, t_2]$ defined in Lemma 2.9 in \cite{Grigoryan2024}. Then using the same arguments as in Lemma 2.9 in \cite{Grigoryan2024}, we obtain \begin{equation}\label{beforecomp}\left[\int_{\Omega}{F_\varepsilon(u)\eta^{p}}\right]_{t_{1}}^{t_{2}}\leq \int_{Q}{-\langle|\nabla u^{q}|^{p-2}\nabla u^{q}, \nabla(f_\varepsilon(u^{q})\eta^{p})\rangle +pF_\varepsilon(u)\eta^{p-1}\partial_{t}\eta } ,\end{equation} where $Q=\Omega\times [t_1, t_2]$ and $$F_{\varepsilon}(s)=\int_{0}^{s}f_\varepsilon(r^q)dr.$$
We have $$\langle|\nabla u^{q}|^{p-2}\nabla u^{q}, \nabla(f_\varepsilon(u^{q})\eta^{p})\rangle=a\frac{|\nabla u^{q}|^{p}\eta^p}{(u^{q}+\varepsilon)^{1-a}}+pf_\varepsilon(u^{q})\eta^{p-1}|\nabla u^{q}|^{p-2}\langle \nabla u^{q}, \nabla \eta\rangle.$$
By Young's inequality, we obtain, for any $b>0$, $$f_\varepsilon(u^{q})\eta^{p-1}|\nabla u^{q}|^{p-1} |\nabla \eta|\leq  b^{p/(p-1)}\frac{|\nabla u^{q}|^{p}\eta^p}{(u^{q}+\varepsilon)^{1-a}}+\frac{1}{b^{p}}f_\varepsilon(u^{q})^p(u^{q}+\varepsilon)^{(1-a)(p-1)}|\nabla \eta|^{p}.$$ Since $0<a<1$, we see that $$f_\varepsilon(u^{q})^p(u^{q}+\varepsilon)^{(1-a)(p-1)}\leq Cu^{q(a+p-1)}=Cu^{\sigma}.$$ Substituting this into (\ref{beforecomp}) yields $$\left[\int_{\Omega}{F_\varepsilon(u)\eta^{p}}\right]_{t_{1}}^{t_{2}}+(a-pb^{p/(p-1)})\int_{Q}\frac{|\nabla u^{q}|^{p}\eta^p}{(u^{q}+\varepsilon)^{1-a}}\leq \int_{Q}{C u^{\sigma}|\nabla \eta|^{p} +pF_\varepsilon(u)\eta^{p-1}\partial_{t}\eta }.$$
Choosing $b>0$ small enough so that $a-pb^{p/(p-1)}>0$ we obtain  \begin{equation}\label{choosingbsmall}\left[\int_{\Omega}{F_\varepsilon(u)\eta^{p}}\right]_{t_{1}}^{t_{2}}+c\int_{Q}\frac{|\nabla u^{q}|^{p}\eta^p}{(u^{q}+\varepsilon)^{1-a}}\leq \int_{Q}{C u^{\sigma}|\nabla \eta|^{p} +pF_\varepsilon(u)\eta^{p-1}\partial_{t}\eta }.\end{equation}
Note that $f_{\varepsilon}(u^q)\to u^{qa}=u^{\lambda-1}$ and thus $$F_{\varepsilon}(u)\to u^{\lambda}/\lambda \quad\textnormal{as}~ \varepsilon\to 0.$$
Consider also the function $$g_\varepsilon(s)=\int_{0}^{s}\frac{dr}{(r+\varepsilon)^{(1-a)/p}}.$$
Then $g^{\prime}(s)=(s+\varepsilon)^{-(1-a)/p}$ so that  $$|\nabla g_\varepsilon(u^q)|^p=\frac{|\nabla u^{q}|^p}{(u^{q}+\varepsilon)^{1-a}}$$ and $$\lim_{\varepsilon\to 0}g_\varepsilon(u^q)=Cu^{q(a+p-1)/p}=Cu^{\sigma/p}.$$ Therefore, $$\int_Q |\nabla u^{\sigma/p}|^p\eta^p\leq C \liminf_{\varepsilon\to 0}\int_{Q}\frac{|\nabla u^{q}|^{p}\eta^p}{(u^{q}+\varepsilon)^{1-a}}.$$
It follows from (\ref{choosingbsmall}) that $$\left[\int_{\Omega}{u^{\lambda}\eta^{p}}\right]_{t_{1}}^{t_{2}}+c\int_{Q}|\nabla u^{\sigma/p}|^p\eta^p\leq \int_{Q}{C u^{\sigma}|\nabla \eta|^{p} +pu^{\lambda}\eta^{p-1}\partial_{t}\eta }.$$
Since $$|\nabla(u^{\sigma/p}\eta)|^{p}\leq C|\nabla u^{\sigma/p}|^p\eta^p+Cu^{\sigma}|\nabla \eta|^p,$$ we conclude (\ref{vetacor}).
\end{proof}
	
	\begin{lemma}[Lemma 2.9 \cite{Grigoryan2024}]
		\label{monl1}
		Let $u=u\left( x,t\right) $ be a non-negative bounded subsolution to \emph{(\ref{dtv})} in $M\times I$  with non-negative initial function $u(\cdot, 0)=u_0\in L^{1}(M)\cap L^{\infty}(M)$. If $\sigma\geq 1$, including $\sigma=\infty$, then the function 
		\begin{equation*}
			t\mapsto \left\Vert u(\cdot ,t)\right\Vert _{L^{\sigma}(M)}
		\end{equation*}%
		is monotone decreasing in $I$.
	\end{lemma}

\section{Integral estimates for subsolutions}
\label{intestsec}

Let $M$ be a connected Riemannian manifold.
Let $d$ be the geodesic distance on $M$. For any $%
x\in M$ and $r>0$, denote by $B(x,r)$ the geodesic ball of radius $r$
centered at $x$, that is,%
\begin{equation*}
	B(x,r)=\left\{ y\in M:d(x,y)<r\right\}.
\end{equation*}
We denote its volume by $$V(x, r)=\mu(B(x, r)).$$

The Davies-Gaffney type estimate for solutions of (\ref{dtv}) was proved in the case $\delta=0$ in \cite{surig2024sharp}.

Here, we extend it to $\delta\geq 0$, that is, $q(p-1)\geq 1$.

\subsection{Integral maximum principle}
\label{intmaxp}

\begin{lemma}\label{integdec}
	Let $u$ be a non-negative bounded subsolution of (\ref{dtv}) in $M\times [0, \infty)$, with non-negative initial function $u(\cdot, 0)=u_0\in L^{1}(M)\cap L^{\infty}(M)$. Fix some $\sigma\geq pq$ and set $\lambda=\sigma-\delta$. Let $\xi(x, t)$ be a non-positive locally Lipschitz function in $M\times[0, \infty)$ and assume that the partial derivative $\partial_{t}\xi$ satisfies the inequality \begin{equation}\label{condfordec}\partial_{t}\xi+C||u_0||_{L^{\infty}}^{\delta}|\nabla \xi|^{p}\leq 0,\end{equation} for some positive constant $C$ depending on $p, q, \sigma$.
	Then the function \begin{equation}\label{defJ}J(t)=\int_{M}u^{\lambda}(\cdot, t)e^{\xi(\cdot, t)}\end{equation} is non-increasing in $t\in [0, \infty)$.
\end{lemma}

\begin{proof}
	Since $\xi$ is non-positive and $\sigma\geq pq$, where the latter implies that $\lambda\geq q+1$, we see that the integral in (\ref{defJ}) is finite. 
	Let $\varphi$ be a cut-off function of some open geodesic ball $B^{\prime}$ such that $\varphi$ has compact support in some larger ball $B$. Note that the balls are precompact by the completeness of $M$. Then set $\eta(x, t)=e^{\frac{\xi(x, t)}{p}}\varphi(x)$ so that $$\eta^{p}=e^{\xi}\varphi^{p}.$$
	
	By the Caccioppoli type inequality (\ref{vetacor}), we have for $0\leq t_{1}<t_{2}<\infty$,
	\begin{equation*}
		\left[ \int_{B }u^{\lambda }\eta ^{p}\right] _{t_{1}}^{t_{2}}\leq \int_{Q}\left[ p\eta ^{p-1}\partial _{t}\eta
		+c_{2}||u||_{L^{\infty}}^{\delta}\left\vert \nabla \eta \right\vert ^{p}\right]u^{\lambda },
	\end{equation*} where $Q=B\times [t_{1}, t_{2}]$.
	Noticing that $$|\nabla\eta |^{p}=\left|e^{\frac{\xi}{p}}\nabla \varphi+p^{-1}e^{\frac{\xi}{p}}\varphi\nabla \xi\right|^{p}\leq 2^{p-1}\left(e^{\xi}|\nabla \varphi|^{p}+p^{-p}e^{\xi}\varphi^{p}|\nabla \xi|^{p}\right)$$ and $$p\eta ^{p-1}\partial _{t}\eta=\varphi^{p}e^{\xi}\partial_{t}\xi,$$ we obtain
	\begin{align}\left[ \int_{B }{u^{\lambda }e^{\xi}\varphi^{p}}\right] _{t_{1}}^{t_{2}}\leq C^{\prime}\int_{Q}\left[\varphi^{p}\partial_{t}\xi+||u||_{L^{\infty}}^{\delta}|\nabla\varphi|^{p}+C||u||_{L^{\infty}}^{\delta}\varphi^{p}|\nabla \xi|^{p}\right]e^{\xi}u^{\lambda}.\label{beforecond}\end{align}
	Hence, it follows from (\ref{condfordec}) and Lemma \ref{monl1}, that $$\left[ \int_{B }{u^{\lambda }e^{\xi}\varphi^{p}}\right] _{t_{1}}^{t_{2}}\leq C^{\prime}||u_0||_{L^{\infty}}^{\delta} \int_{Q}u^{\lambda}e^{\xi}|\nabla\varphi |^{p}.$$ Finally, we get, by sending $B\to M$ and using that $\varphi\to 1$ and $|\nabla \varphi|\to 0$ as $B\to M$, $$\left[ \int_{M }{u^{\lambda }e^{\xi}}\right] _{t_{1}}^{t_{2}}\leq 0,$$ which finishes the proof.
\end{proof}

\subsection{Davies-Gaffney type inequality}
\label{daviesty}

For any subset $A$ of $M$ and any $r>0$, set $$A_{r}=\{x\in M :d(x, A)<r\}.$$ In the next lemma, we will also use $$A_{r}^{c}=(A_{r})^{c}=\{x\in M:d(x, A)\geq r\}.$$

\begin{lemma}\label{choicofxi}
	Let $u$ be a non-negative bounded subsolution of (\ref{dtv}) in $M\times [0, \infty)$ with non-negative initial function $u(\cdot, 0)=u_0\in L^{1}(M)\cap L^{\infty}(M)$ and $A\subset M$ be measurable. Suppose that again $\sigma\geq pq$ and set $\lambda=\sigma-\delta$. Then, for all $r, t>0$, \begin{equation}\label{upperintarc} \int_{A_{r}^{c}}u^{\lambda}(\cdot, t)\leq \int_{A^{c}}u_{0}^{\lambda}+\exp\left(-\zeta\left(\frac{r}{(||u_0||_{L^{\infty}}^{\delta}t)^{1/p}}\right)^{\frac{p}{p-1}}\right)\int_{A}u_{0}^{\lambda},\end{equation} where $u_{0}=u(\cdot, 0)$ and $\zeta>0$ depends on $p, q, \sigma$.
\end{lemma}

\begin{proof}
	Fix some $s>t$ and define, for all $x\in M$ and $0\leq \tau< s$, the function $$\xi(x,\tau)=-\zeta\left(\frac{d(x, A_{r}^{c})}{[||u_0||_{L^{\infty}}^{\delta}(s-\tau)]^{1/p}}\right)^{\frac{p}{p-1}}.$$
	For such $\tau$, let us also define $$J(\tau)=\int_{M}u^{\lambda}(\cdot, \tau)e^{\xi(\cdot, \tau)} .$$
	Note that $\xi$ is non-positive and locally Lipschitz in $M\times[0, s)$.
	Using that $|\nabla d(x, A_{r}^{c})|\leq 1$, we get $$\left|\nabla \xi(x, \tau)\right|^{p}\leq C\zeta^{p}||u_0||_{L^{\infty}}^{-\delta p/(p-1)} \frac{d(x, A_{r}^{c})^{\frac{p}{p-1}}}{(s-\tau)^{\frac{p}{p-1}}}.$$
	Since $$\partial_{\tau}\xi=-C\zeta||u_0||_{L^{\infty}}^{-\delta/(p-1)}\frac{d(x, A_{r}^{c})^{\frac{p}{p-1}}}{(s-\tau)^{\frac{p}{p-1}}},$$ we can choose $\zeta>0$ so that the partial derivative $\partial_{\tau}\xi$ satisfies the condition $$\partial_{\tau}\xi+C||u_0||_{L^{\infty}}^{\delta}|\nabla \xi|^{p}\leq 0.$$ Hence, we obtain from Lemma \ref{integdec} that \begin{equation}\label{Jnondeczero}J(t)\leq J(0).\end{equation}
	As $d(x, A_{r}^{c})\geq r$ for all $x\in A$, we have $$\xi(x, 0)\leq -\zeta\left(\frac{r}{(||u_0||_{L^{\infty}}^{\delta}s)^{1/p}}\right)^{\frac{p}{p-1}}~ \textnormal{for all}~x\in A.$$
	Together with the fact that $\xi(x, 0)\leq0$ for all $x\in M$, we therefore obtain that \begin{equation}\label{upperJzero}J(0)=\int_{A^{c}}u_{0}^{\lambda}e^{\xi(\cdot, 0)}+\int_{A}u_{0}^{\lambda}e^{\xi(\cdot, 0)}\leq \int_{A^{c}}u_{0}^{\lambda}+\exp\left(-\zeta\left(\frac{r}{(||u_0||_{L^{\infty}}^{\delta}s)^{1/p}}\right)^{\frac{p}{p-1}}\right)\int_{A}u_{0}^{\lambda}.\end{equation}
	
	Since $\xi(x, t)=0$ for $x\in A_{r}^{c}$, it follows that $$J(t)\geq\int_{A_{r}^{c}}u^{\lambda}(\cdot, t)e^{\xi(\cdot, t)}=\int_{A_{r}^{c}}u^{\lambda}(\cdot, t).$$
	Hence, combining this with (\ref{Jnondeczero}) and (\ref{upperJzero}), we see that $$\int_{A_{r}^{c}}u^{\lambda}(\cdot, t)\leq \int_{A^{c}}u_{0}^{\lambda}+\exp\left(-\zeta\left(\frac{r}{(||u_0||_{L^{\infty}}^{\delta}s)^{1/p}}\right)^{\frac{p}{p-1}}\right)\int_{A}u_{0}^{\lambda}.$$
	Sending now $s\to t+$, we finally obtain (\ref{upperintarc}).
\end{proof}
	
\section{Conservation of mass}\label{seccons}

Here, we again assume that $\delta\geq0$ (cf. (\ref{delta})).

\begin{theorem}\label{expmass}
Assume that, for some $x_0\in M$ and all large enough $r>0$, \begin{equation}\label{expvol}V(x_0, r)\leq \exp\left(C r^{\frac{p}{p-1}}\right),\end{equation} where $C$ is a positive constant.
Let $u$ be a non-negative bounded solution to (\ref{dtv}) in $M\times [0, \infty)$ with non-negative initial function $u(\cdot, 0)=u_0\in L^{1}(M)\cap L^{\infty}(M)$. Then, for all $t>0$, \begin{equation}\label{masscons}\int_M u(x, t)=\int_M u_0(x).\end{equation}
\end{theorem}

\begin{proof}
Let us first assume that $u_0$ has a compact support, say $\textnormal{supp}~u_0\subset B(x_0, r_0)$. Choose $r>4r_0$. Then Lemma \ref{daviesty} with $\sigma=pq$ yields, for all $s>0$, \begin{equation}\label{DGI}
\int_{B(x_0, r/2)^{c}}u^{q+1}(\cdot, s)\leq  \exp\left(-c\left(\frac{r}{(||u_0||_{L^{\infty}}^{\delta}s)^{1/p}}\right)^{\frac{p}{p-1}}\right)\int_{B(x_0, r_0)}u_0^{q+1}.
\end{equation}
Define $A_r=B(x_0, 2r)\setminus B(x_0, r)$ and let $\eta_r$ be a cut-off function of $A_r$ in $B(x_0, 3r)\setminus B(x_0, r/2)$, so that $$|\nabla \eta_r|\leq \frac{C}{r}.$$ Since $r>4r_0$, we see that $u_0=0$ in $\textnormal{supp}~\eta_r$.
It follows from (\ref{vetacor}) that, for all $t>0$, $$\int_{0}^{t}\int_{A_r}\left\vert \nabla  u^{q }  \right\vert
^{p}\leq \frac{C}{r^{p}}\int_{0}^{t}\int_{\textnormal{supp}~\eta_r}u^{pq}\leq \frac{C}{r^{p}} ||u_0||_{L^{\infty}}^{\delta}\int_{0}^{t}\int_{\textnormal{supp}~\eta_r}u^{q+1}.$$ Hence, combining this with (\ref{DGI}), we obtain \begin{equation}\label{gradest}\int_{0}^{t}\int_{A_r}\left\vert \nabla  u^{q }  \right\vert^{p}\leq \frac{Ct}{r^{p}}||u_0||_{L^{\infty}}^{\delta}\int_{B(x_0, r_0)}u_0^{q+1}\exp\left(-c\left(\frac{r}{(||u_0||_{L^{\infty}}^{\delta}t)^{1/p}}\right)^{\frac{p}{p-1}}\right)\end{equation}

Now let us take another cut-off function $\psi_r$ of $B(x_0, r)$ in $B(x_0, 2r)$. Then testing in (\ref{defvonweaksolq}) with $\psi_r$ we get, for all $t>0$, $$\left[\int_{B(x_0, 2r)}{u\psi_r }\right]_{0}^{t}=-\int_{0}^{t}{\int_{A_r}{|\nabla u^{q}|^{p-2}\langle\nabla u^{q},\nabla \psi_r\rangle}}.$$ Setting $$I_r(t)=\left|\int_{B(x_0, 2r)}{u(\cdot, t)\psi_r} -\int_{B(x_0, 2r)}u_0\psi_r \right|,$$ we obtain from Hölder's inequality $$I_r(t)\leq \frac{C}{r}\left(\int_{0}^{t}\int_{A_r}\left\vert \nabla  u^{q }  \right\vert^{p}\right)^{(p-1)/p}\left(t\mu(A_r)\right)^{1/p}.$$
Hence, by (\ref{gradest}), $$I_r(t)\leq \frac{Ct}{r^{p}}||u_0||_{L^{\infty}}^{\delta(p-1)/p}||u_0||_{L^{q+1}}^{(q+1)(p-1)/p}V(x_0, 2r)^{1/p}\exp\left(-c\left(\frac{r}{(||u_0||_{L^{\infty}}^{\delta}t)^{1/p}}\right)^{\frac{p}{p-1}}\right).$$
Thus, we obtain from the volume bound (\ref{expvol}), $$ I_r(t)\leq \frac{Ct}{r^{p}}||u_0||_{L^{\infty}}^{\delta(p-1)/p}||u_0||_{L^{q+1}}^{(q+1)(p-1)/p}\exp\left(-\left[\frac{c}{(||u_0||_{L^{\infty}}^{\delta}t)^{1/(p-1)}}-c^{\prime}\right]r^{\frac{p}{p-1}}\right).$$
Now let us choose $t_0=c_0 ||u_0||_{L^{\infty}}^{-\delta}$, where $c_0$ is sufficiently small. Thus, for every $t< t_0$, we have $I_r(t)\to 0$ for $r\to \infty$, that means, for every $t< t_0$, \begin{equation}\label{forsmallt}\int_{M}u(\cdot, t)=\int_{M}u_0.\end{equation}

Let now $u_0\in L^{1}(M)\cap L^{\infty}(M)$ be arbitrary, that is, it might have unbounded support. Then let us define $u_{0, n}=u_0 1_{B(x_0, n)}$ and let $u_n$ be a corresponding solution. Then, by (\ref{forsmallt}), for every $n$ and every $t< t_0$, $$\int_{M}u_n(\cdot, t)=\int_{M}u_{0, n}.$$ It follows from the comparison principle Lemma 2.5 in \cite{surig2026existence}, that $$\int_{M}|u_n(\cdot, t)-u(\cdot, t)|\leq \int_{M}|u_{0, n}-u_0|\to 0.$$ Therefore, $$\int_{M}u(\cdot, t)=\lim_{n\to \infty}\int_{M}u_n(\cdot, t)=\lim_{n\to \infty}\int_{M}u_{0, n}=\int_{M}u_0,$$ which proves (\ref{forsmallt}) also in this case.

Let us now extend this to all $t>0$. By Lemma \ref{monl1} we have $u(\cdot, t_0)\in L^{1}(M)\cap L^{\infty}(M)$. Since $c_0 ||u(t_0)||_{L^{\infty}}^{-\delta}\geq c_0 ||u_0||_{L^{\infty}}^{-\delta}=t_0$, it follows from (\ref{forsmallt}) that, for all $t< t_0$, $$\int_{M}u(\cdot, t+t_0)=\int_{M}u(\cdot, t_0).$$ Hence, for all $t<2t_0$, $$\int_{M}u(\cdot, t)=\int_{M}u_0.$$ Iterating this argument we obtain (\ref{forsmallt}) for all $t>0$, which finishes the proof.
\end{proof}

\begin{conjecture}
The statement of Theorem \ref{expmass} holds under the assumption that, for some $x_0\in M$ and all large enough $r>0$, \begin{equation}\label{conjecture}\int^{\infty}\frac{r^{p-1}dr}{(\log V(x_0, r))^{p-1}}=\infty.\end{equation}	
\end{conjecture}

Clearly, (\ref{expvol}) implies (\ref{conjecture}).

\subsection{Degenerate case}

Let us assume that $\delta>0$, that is, $q(p-1)>1$.
	
\begin{proposition}\label{propfps}
Assume that $M$ has non-negative Ricci curvature. Let $u$ be a non-negative bounded solution to (\ref{dtv}) in $M\times [0, \infty)$ with non-negative initial function $u(\cdot, 0)=u_0\in L^{1}(M)\cap L^{\infty}(M)$. Then, for all $t>0$, \begin{equation}\label{massconsdeg}\int_M u(x, t)=\int_M u_0(x).\end{equation}
\end{proposition}

\begin{proof}
Let us first assume that $u_0$ has a compact support. Then it is well-known (cf. \cite{grigor2024finite}) that $u$ has a finite propagation speed, that is, for every $t>0$, $$\textnormal{supp}~u(\cdot, t)\subset K(t),$$ where $K(t)$ is a precompact neighborhood of $\textnormal{supp}~u_0$ depending on $t$.
Let us now choose $\psi$ to be a cut-off function of $K(t)$ in $M$.
Testing in (\ref{defvonweaksolq}) with such $\psi$ we get, for all $t>0$, $$\left[\int_{M}{u\psi }\right]_{0}^{t}=-\int_{0}^{t}{\int_{M}{|\nabla u^{q}|^{p-2}\langle\nabla u^{q},\nabla \psi\rangle}}.$$
Since $\nabla \psi=0$ in $K(t)$ and $u(\cdot, s)=0$ in $M\setminus K(t)$, we get $$\int_{M}{u(\cdot, t)\psi }=\int_{M}{u_0\psi }.$$ 
Using again that $\psi=1$ on $K(t)$, we also obtain, for all $t>0$, $$\int_{M}{u(\cdot, t) }=\int_{M}{u_0}.$$ Using the same arguments as in the proof of Theorem \ref{expmass} we can extend this so that (\ref{massconsdeg}) holds for arbitrary $u_0\in L^{1}(M)\cap L^{\infty}(M)$ and all $t>0$. 
\end{proof}


\subsection{Singular case}

Let us assume that $D:=-\delta=1-q(p-1)>0$, that is, $q(p-1)<1$.

\begin{theorem}\label{expmasssing}
	Assume that, for some $x_0\in M$ and all large enough $r>0$, \begin{equation}\label{expvolsing}V(x_0, r)\leq Cr^{N},\end{equation} where $C, N$ are positive constants.
	Suppose that \begin{equation}\label{pnD}p>ND\end{equation} and let $u$ be a non-negative bounded solution to (\ref{dtv}) in $M\times [0, \infty)$ with non-negative initial function $u(\cdot, 0)=u_0\in L^{1}(M)\cap L^{\infty}(M)$. Then, for all $t>0$, \begin{equation}\label{massconssing}\int_M u(x, t)=\int_M u_0(x).\end{equation}
\end{theorem}

\begin{proof}
Let us first assume that $u_0$ has a compact support. For $0<\varepsilon<\min(D, q)$ let $\sigma=1+\varepsilon-D>q(p-1)$, and set $\lambda=\sigma+D=1+\varepsilon$. Let $\eta_r$ be a cut-off function of $A_r=B(x_0, 2r)\setminus B(x_0, r)$ in $B(x_0, 3r)\setminus B(x_0, r/2)$, so that $$|\nabla \eta_r|\leq \frac{C}{r}.$$ Choosing $r$ large enough we have that the support of $\eta_r$ is disjoint from the support of $u_0$ and we obtain from (\ref{vetacor}) $$\int_{0}^{t}\int_{A_r}\left\vert \nabla  u^{\sigma/p }  \right\vert
^{p}\leq \frac{C}{r^{p}}\int_{0}^{t}\int_{B(x_0, 4r)}u^{\sigma}.$$
As $0<\sigma<1$, we obtain by Hölder's inequality and Lemma \ref{monl1} $$\int_{B(x_0, 4r)}u^{\sigma}\leq \left(\int_{B(x_0, 4r)}u\right)^{\sigma}V(x_0, 4r)^{1-\sigma}\leq ||u_0||_{L^1}^{\sigma}V(x_0, 4r)^{1-\sigma}.$$
Applying (\ref{expvolsing}) we therefore get \begin{equation}\label{intsigmap}\int_{0}^{t}\int_{A_r}\left\vert \nabla  u^{\sigma/p }  \right\vert
^{p}\leq Ct||u_0||_{L^1}^{\sigma} r^{-p+N(1-\sigma)}.\end{equation}
Since $\nabla  u^{\sigma/p } =\frac{\sigma}{pq}u^{\sigma/p-q}\nabla u^{q}$ we have $$|\nabla u^{q}|^{p-1}=Cu^{(q-\sigma/p)(p-1)}|\nabla u^{\sigma/p}|^{p-1}.$$ Hence, again using Hölder's inequality \begin{equation}\label{substi}\int_{0}^{t}\int_{A_r}\left\vert \nabla  u^{q }  \right\vert^{p-1}\leq C\left(\int_{0}^{t}\int_{A_r}\left\vert \nabla  u^{\sigma/p }  \right\vert
^{p}\right)^{(p-1)/p}\left(\int_{0}^{t}\int_{A_r}u^{a}\right)^{1/p},\end{equation} where $$a=(pq-\sigma)(p-1)=(q-\varepsilon)(p-1).$$ We have $0<a<1$ as $a\to 1-D$ for $\varepsilon\to0$. Thus, again by Hölder's inequality and (\ref{expvolsing}) \begin{equation}\label{intaa}\int_{A_r}u^{a}\leq  ||u_0||_{L^1}^{a}V(x_0, 2r)^{1-a}\leq C||u_0||_{L^1}^{a}r^{N(1-a)}.\end{equation}
Thus, substituting (\ref{intsigmap}) and (\ref{intaa}) into (\ref{substi}), we get $$\int_{0}^{t}\int_{A_r}\left\vert \nabla  u^{q }  \right\vert^{p-1}\leq Ct ||u_0||_{L^1}^{[\sigma(p-1)+a]/p}r^{\frac{p-1}{p}[-p+N(1-\sigma)]+\frac{N}{p}(1-a)}.$$
Observe that $$[\sigma(p-1)+a]/p=1-D$$ and $$\frac{p-1}{p}[-p+N(1-\sigma)]+\frac{N}{p}(1-a)=1-p+ND.$$
Hence, \begin{equation}\label{gradestimate}\int_{0}^{t}\int_{A_r}\left\vert \nabla  u^{q }  \right\vert^{p-1}\leq Ct ||u_0||_{L^1}^{1-D}r^{1-p+ND}.\end{equation}

Now let us take another cut-off function $\psi_r$ of $B(x_0, r)$ in $B(x_0, 2r)$. Then testing in (\ref{defvonweaksolq}) with $\psi_r$ we get, for all $t>0$, $$\left[\int_{B(x_0, 2r)}{u\psi_r }\right]_{0}^{t}=-\int_{0}^{t}{\int_{A_r}{|\nabla u^{q}|^{p-2}\langle\nabla u^{q},\nabla \psi_r\rangle}}.$$ Hence,  $$\left|\int_{B(x_0, 2r)}{u(\cdot, t)\psi_r} -\int_{B(x_0, 2r)}u_0\psi_r \right|\leq \frac{C}{r}\int_{0}^{t}\int_{A_r}\left\vert \nabla  u^{q }  \right\vert^{p-1}.$$ Therefore, by (\ref{gradestimate}), $$\left|\int_{B(x_0, 2r)}{u(\cdot, t)\psi_r} -\int_{B(x_0, 2r)}u_0\psi_r \right|\leq Ct ||u_0||_{L^1}^{1-D}r^{-p+ND}.$$ By (\ref{pnD}) we can send $r\to \infty$ and obtain (\ref{massconssing}) in the case when $u_0$ has a compact support.
Finally, using the same method as in the proof of Theorem \ref{expmass}, we obtain (\ref{massconssing}) for arbitrary $u_0\in L^{1}(M)\cap L^{\infty}(M)$.
\end{proof}

\begin{example}
Since $\mathbb{R}^{n}$ satisfies (\ref{expvolsing}) with $N=n$, we obtain from Theorem \ref{expmasssing} mass conservation in the case $p>nD$.
\end{example}

\subsection{$L^{p-1}$-Liouville property}

\begin{proposition}\label{propLiou}
	Assume that $M$ is connected and that (\ref{expvol}) holds. Then every non-negative lower semi-continuous function $v\in L^{p-1}(M)$, with $v\in W_{loc}^{1, p}(M)$ satisfying $-\Delta_pv\geq 0$ weakly in $M$, is constant.
\end{proposition}

\begin{proof}
For $k>0$, let us define $v_k=\min(v, k)$ and $w_k=v_k^{p-1}$. Then $v_k\in W_{loc}^{1, p}(M)$ and $-\Delta_pv_k\geq 0$ weakly in $M$. Also, $w_k\leq k^{p-1}$ and $$\int_M w_k\leq \int_M v^{p-1}<\infty,$$ so that $w_k\in L^{1}(M)\cap L^{\infty}(M)$. Let $u_k=u_k(x, t)$ be a weak solution of the Trudinger equation $\partial _{t}u=\Delta _{p}u^{\frac{1}{p-1}}$ with initial function $w_k$. Then it follows from Theorem \ref{expmass} that, for all $t>0$, \begin{equation}\label{appconofmass}\int_M u_k(\cdot, t)=\int_M w_k.\end{equation}
Since $w_k^{1/(p-1)}=v_k$, we have  $-\Delta_pw_k^{1/(p-1)}\geq 0$ weakly in $M$, and we see that the time-independent function $w_k$ is a weak supersolution of the Trudinger equation. It then follows from the comparison principle Lemma 2.5 in \cite{surig2026existence}, that $u_k(\cdot, t)\leq w_k$ a.e. in $M$.
Together with (\ref{appconofmass}) this yields $u_k(\cdot, t)=w_k$ a.e. in $M$. Hence, by the solution analogue of (\ref{defvonweaksolq}) with test function $\psi(x, t)=\varphi(x)\eta(t)$, where $\varphi\in C_0^{\infty}(M)$ and $\eta$ is a smooth cut-off function in $[t_1, t_2]$, $t_1<t_2$, we get $$\int_{M}{|\nabla v_k|^{p-2}\langle\nabla v_k,\nabla \varphi\rangle}= 0,$$ that is, $-\Delta_pv_k= 0$ weakly in $M$.
Assume now that $v$ is non-constant. Let $k>0$ be such that there exist $x, y\in M$ with $v(x)<k<v(y)$. Since $v$ is lower semi-continuous, the set $\{v>k\}$ is open and non-empty and on this set $v_k=k$. Hence, the $p$-harmonic function $v_k$ attains its maximum on an non-empty open set. The maximum principle for $p$-harmonic functions (cf. \cite{8b811f47b8724a088f74f7843726f67f}) then gives $v_k\equiv k$ on $M$. However, we have $v_k(x)=v(x)<k$ which gives a contradiction.
\end{proof}

\begin{remark}\label{remholo}
It is conjectured by I. Holopainen \cite{holopainen2000sharp} that (\ref{conjecture}) implies the $L^{p-1}$-Liouville-property of Proposition \ref{propLiou} (see also \cite{pigola2006some}). Hence, our result Proposition \ref{propLiou} gives a partial answer to that conjecture. In the linear case $p=2$, the conjecture was proved by A. Grigor'yan \cite{grigor1989stochastically}. 
\end{remark}

\section{Finite extinction time}\label{secextinct}

\begin{definition}
We say that a function $u=u(x, t)$ has a \textit{finite extinction time} in $M$ if there exists $T>0$ such that $$u(\cdot, t)\equiv0\quad \text{for all}~t\geq T.$$	
\end{definition}

Clearly, if $u$ has a finite extinction time, $u$ cannot have the conservation of mass property.

In this Section we always assume that $D>0$ holds, that is, $q(p-1)<1$.

\begin{proposition}\label{thmfinexint}
	Suppose that the manifold $M$ admits the Sobolev inequality \begin{equation}\label{Sobolevinequality}\left(\int_{M}{|f|^{p\kappa} d\mu}\right)^{1/\kappa}\leq C\int_{M}{|\nabla f|^{p}d\mu}\quad \forall~f\in W^{1, p}(M),\end{equation}  with some Sobolev exponent $\kappa>1$. Let $\rho$ be a positive function on $M$ and let $u$ be a non-negative bounded weak solution of $\rho\partial _{t}u=\Delta _{p}u^{q}$ in $M\times [0, \infty)$, with initial function $u(\cdot, 0)=u_{0}\in L^{\zeta}(M, \rho d\mu)\cap L^{\infty}(M)$ for some $\zeta\geq1$. If $\rho$ is bounded and satisfies \begin{equation}\label{finitecondint}\left|\left|\rho\right|\right|_{L^{\theta}(M, \mu)}<\infty,\end{equation} where \begin{equation}\label{thetamax}\theta=\left\{\begin{array}{ll}
			\frac{\kappa}{\kappa-1-\frac{D}{\zeta-D}}, & \text{if }\kappa\geq \frac{\zeta}{\zeta-D}~\textnormal{and}~\zeta> 1, \\ 
			\infty, & \text{if }\kappa< \frac{\zeta}{\zeta-D}.%
		\end{array}%
		\right.\end{equation} Then $u$ has a finite extinction time in $M$.
\end{proposition}
	
\begin{proof}
In Theorem 1.1 in \cite{surig2024finite} the statement is proved for \begin{equation}\label{defthetaint}\theta=\frac{\kappa}{ \kappa-1-\frac{D}{\sigma}}\end{equation} where \begin{equation}\label{sigmalargeintreal}\sigma=\sigma_{old}= \max\left(pq, \zeta-D, \frac{D}{\kappa-1}\right)\end{equation} and the condition $\sigma_{old}\geq pq$ comes from a Caccioppoli type inequality in this paper. In the Caccioppoli type inequality Lemma \ref{Lem1} from the present paper, this condition is replaced by $\sigma>q(p-1)=1-D$. Hence, using the same arguments as in \cite{surig2024finite} we obtain Proposition \ref{thmfinexint} with $\theta$ as in (\ref{defthetaint}) and $\sigma$ given by \begin{equation}\label{sigmaneu}\sigma= \max\left(\zeta-D, \frac{D}{\kappa-1}\right)\end{equation} if $\sigma>1-D$. Hence, if $\zeta>1$, we clearly have $\sigma\geq \zeta-D>1-D$. In the case $\zeta=1$, it follows from (\ref{sigmaneu}) that $\sigma>1-D$ holds if and only if $\frac{D}{\kappa-1}>1-D$, which is equivalent to $\kappa< \frac{1}{1-D}$. Thus, in either case in (\ref{thetamax}) $\sigma>1-D$ which allows to apply Lemma \ref{Lem1}.
\end{proof}	

\begin{example}\label{exampleRn}
Let $M=\mathbb{R}^{n}$, $\zeta=1$ and $n>p$. Then inequality (\ref{Sobolevinequality}) holds with $\kappa=\frac{n}{n-p}$. It follows from (\ref{thetamax}) that we have $\theta=\infty$ if $\kappa< \frac{1}{1-D}$, which is in this case equivalent to \begin{equation}\label{nDp}nD>p.\end{equation} Then  (\ref{finitecondint}) holds for $\rho\equiv1$. Hence, we conclude a finite extinction time for solutions of the Leibenson equation (\ref{dtv}) in $\mathbb{R}^{n}$ assuming (\ref{nDp}). Also, there is a positive fundamental solution of (\ref{dtv}) in $\mathbb{R}^{n}$ (\textit{Barenblatt solution}) that decays as $|x|^{-p/D}$ whence it is not in $L^{1}$ if (\ref{nDp}) holds (cf. \cite{grigor2026sharp, grigor2025upper}).
\end{example}

\begin{example}\label{exampleHn}
Let $M=\mathbb{H}^{n}$, $\zeta=1$ and $n>p$. It is well-known \cite{chavel1984eigenvalues} that in $\mathbb{H}^{n}$ $$\int_{\mathbb{H}^{n}}{|f|^{p} d\mu}\leq C\int_{\mathbb{H}^{n}}{|\nabla f|^{p}d\mu}\quad \forall~f\in W^{1, p}(\mathbb{H}^{n})$$ and $$\left(\int_{\mathbb{H}^{n}}{|f|^{\frac{pn}{n-p}} d\mu}\right)^{(n-p)/n}\leq C\int_{\mathbb{H}^{n}}{|\nabla f|^{p}d\mu}\quad \forall~f\in W^{1, p}(\mathbb{H}^{n}).$$
Hence, by interpolation we obtain that (\ref{Sobolevinequality}) is satisfied for any $1<\kappa\leq \frac{n}{n-p}$. In particular, we can choose $1<\kappa<\min\left(\frac{n}{n-p}, \frac{1}{1-D}\right)$, which implies by (\ref{thetamax}) and (\ref{finitecondint}) a finite extinction time for solutions of the Leibenson equation (\ref{dtv}) in $\mathbb{H}^{n}$.
\end{example}

\section{Model manifolds with a fast volume growth}\label{appendix}

Let us construct a model manifold with a faster volume growth than (\ref{expvolint}) in Theorem \ref{expmassint} and a solution $u$ of the Trudinger equation \begin{equation}\label{trudinger}\partial _{t}u=\Delta _{p}u^{\frac{1}{p-1}}\end{equation} that does not have the conservation of mass property. 

\begin{lemma}\label{eigenlemma}
Assume that any non-negative bounded solution $u$ to (\ref{evoeq}) in $M\times [0, T)$, $T\in (0, \infty]$ with non-negative initial function $u(\cdot, 0)=u_0\in L^{1}(M)\cap L^{\infty}(M)$ has the property of \textit{mass conservation}. Fix some $\lambda>0$. Then there is no non-trivial non-negative function $v\in L^{1}(M)\cap L^{\infty}(M)$ with $v^{q}\in W^{1, p}(M)$ such that \begin{equation}\label{eigen}-\Delta_pv^{q}=\lambda v,\end{equation} weakly in $M$.
\end{lemma}

\begin{proof}
Assume that there exists a non-negative non-zero function $v$ satisfying (\ref{eigen}). Let	$u(x, t)=a(t)v(x)$. Then by (\ref{eigen}) $$\Delta_pu^{q}=a(t)^{q(p-1)}\Delta_pv^{q}=-\lambda a(t)^{q(p-1)}v.$$ Since $\partial_tu=a^{\prime}(t)v$, $u$ is a solution to (\ref{evoeq}) if $$a^{\prime}(t)=-\lambda a(t)^{q(p-1)}, \quad a(0)=1.$$ Solving this we get $$a(t)=\left\{ 
\begin{array}{ll}
	e^{-\lambda t}, & \text{if }q(p-1)=1, \\ 
	(1+[q(p-1)-1]\lambda t)^{-\frac{1}{q(p-1)-1}}, & \text{if }q(p-1)>1, \\ 
	(1-[1-q(p-1)]\lambda t)_+^{\frac{1}{1-q(p-1)}}, & \text{if }q(p-1)<1.
\end{array}%
\right.$$ It follows from the mass conservation property that $$\int_Mv=\int_M u(\cdot, 0)=\int_M u_0=a(t)\int_M v,$$ which yields for all $t>0$, $$\int_Mv=0,$$ and thus, $v\equiv 0$ a.e., which finishes the proof.
\end{proof}

Let $M$ be a model manifold, that is $M=(0, +\infty)\times \mathbb{S}^{n-1}$ as topological spaces and $M$ is equipped with the Riemannian metric $ds^{2}$ given by \begin{equation*}ds^{2}=dr^{2}+\psi^{2}(r)d\theta^{2},\end{equation*} where $\psi(r)$ is a smooth positive function on $(0, +\infty)$ and $d\theta^{2}$ is the standard Riemannian metric on $\mathbb{S}^{n-1}$. We define $S(r)=\psi^{n-1}(r)$, which is called the profile of the model manifold.

Assume that, for $r\geq r_0$, $$S(r)= r^{-\beta}\exp\left((p-1)r^{\alpha}+\frac{p-1}{\alpha^{p-1}}\int_{r_0}^{r}s^{-\beta}ds\right),$$ where $$\alpha>\frac{p}{p-1}\quad \textnormal{and}\quad\beta=(\alpha-1)(p-1)>1.$$

By Lemma \ref{eigenlemma}, in order to show that $M$ does not admit the conservation of mass property, it is enough to show the existence of a function $v\in L^{1}(M)\cap L^{\infty}(M)$ satisfying (\ref{eigen}).

Let $v(x)=v(r)=w^{p-1}(r)$, that is, function $v$ depends only on the polar radius $r$. Assume also that $\partial_{r}w\leq 0$, then $$\Delta_{p}w=-\frac{1}{S}\partial _{r}\left( S\left( -\partial
_{r}w\right) ^{p-1}\right)$$ so that (\ref{eigen}) becomes \begin{equation}\label{evoeqmodel}\frac{1}{S}\partial _{r}\left( S\left( -\partial_{r}w\right) ^{p-1}\right)=\lambda w^{p-1}.\end{equation}
Let $w(r)=e^{-r^{\alpha}}$ and note that $$\int_Mw^{p-1}=\int_0^{\infty}w^{p-1}(r)S(r)dr\simeq \int_{r_0}^{\infty}r^{-\beta}<\infty$$ and $$\int_{r_0}^{\infty}|w^{\prime}(r)|^pS(r)dr\simeq \int_{r_0}^{\infty}r^{\alpha-1}e^{-r^{\alpha}}dr<\infty.$$ Let us show that $w$ satisfies (\ref{evoeqmodel}). We have $(-\partial_rw(r))^{p-1}=\alpha^{p-1} r^{\beta}e^{-(p-1)r^{\alpha}}$ and thus, \begin{align*}\partial _{r}\left( S\left( -\partial_{r}w\right) ^{p-1}\right)&=\partial _{r}\left(\alpha^{p-1}\exp\left(\frac{p-1}{\alpha^{p-1}}\int_{r_0}^{r}s^{-\beta}ds\right)\right)\\&=(p-1)r^{-\beta}\exp\left(\frac{p-1}{\alpha^{p-1}}\int_{r_0}^{r}s^{-\beta}ds\right)\\&=(p-1)S(r)w(r)^{p-1},\end{align*} which proves (\ref{evoeqmodel}) with $\lambda=p-1$.

Since $S(r)\simeq r^{-\beta}\exp((p-1)r^{\alpha})$ and thus $V(r)\simeq \exp((p-1)r^{\alpha})$ we obtain indeed a manifold with a faster volume growth than (\ref{expvolint}) in Theorem \ref{expmassint}. In the linear case $p=2$, it follows from Theorem 8.24 in \cite{grigoryan2009heat} that this model manifold is stochastically incomplete.
	
	
	\bibliographystyle{abbrv}
	\bibliography{librarycacc}
	
	\emph{Universit\"{a}t Bielefeld, Fakult\"{a}t f\"{u}r Mathematik, Postfach
		100131, D-33501, Bielefeld, Germany}
	
	\texttt{philipp.suerig@uni-bielefeld.de}
\end{document}